\documentclass{article}
    
    \usepackage[margin=0.7in]{geometry}
    \usepackage[parfill]{parskip}
    \usepackage[utf8]{inputenc}
    
    \usepackage{amsmath,amssymb,amsfonts,amsthm}
\usepackage{thmtools} 
\usepackage{cleveref}
\usepackage{enumitem}

\usepackage{graphicx}
\usepackage{color}

\theoremstyle{plain}
\newtheorem{theo}{Theorem}[section]
\newtheorem{prop}[theo]{Proposition}
\newtheorem{lemma}[theo]{Lemma}
\newtheorem{cor}[theo]{Corollary}

\theoremstyle{definition}

\newcommand{\mc}[1]{\mathcal{#1}}
\newcommand{\mb}[1]{\mathbb{#1}}

\newcommand{\sub}{\subseteq}

\newcommand{\lra}{\leftrightarrow}

\newcommand{\sm}{\setminus}
\newcommand{\ov}{\overline}

\newcommand{\eps}{\varepsilon}
\newcommand{\es}{\emptyset}

\newcommand{\da}{\downarrow}

\newcommand{\lL}{\lambda}

\newcommand{\cE}{\mathcal E}
\newcommand{\OO}{\Omega}

\newcommand{\on}{\operatorname}

\title{On Kahn's flow conjecture}
\author{Peter Keevash\thanks{Mathematical Institute, University of Oxford, UK. Supported by ERC Advanced Grant 883810.}}
\date{}

\begin{document}

\maketitle

\begin{abstract}
We prove Kahn's flow conjecture, which is a strong form of Chv\'atal's conjecture 
on the largest intersecting subfamily of a downset.
\end{abstract}

\section{Introduction}

A family $\mc D\sub \OO:=2^{[n]}$ of subsets of $[n]=\{1,\dots,n\}$ 
is a \emph{downset} if it is closed under taking subsets.  For $i\in[n]$ write $\mc D(i)=\{A\in\mc D:i\in A\}$.  
Chv\'atal conjectured \cite{Chvatal} in 1972 that every intersecting $\mc I\sub\mc D$ satisfies $|\mc I|\le\max_i|\mc D(i)|$.  
After many decades of research, this was recently proved by Chang, Liu and Liu \cite{CLL}.
A more general conjecture of Kleitman \cite{Kleitman}, also proved in \cite{CLL},
states that for any non-negative non-increasing weight function $w$ on $\OO$,
every intersecting $\mc I\sub\mc D$ satisfies $w(\mc I) \le\max_i w(\mc D(i))$.
A flow formulation, shown to be equivalent by Fishburn \cite{F},
is that for any maximal intersecting $\mc B$ there are $c_i\ge0$ with $\sum_i c_i=1$ 
such that $\sum_i c_i x_i$ flows upward to $1_{\mc B}$.
Here, given nonnegative functions $p,q$ on $\OO$, we say that $p$ \emph{flows upward to} $q$
if there are $\beta(x,y)\ge0$, supported on pairs $x\sub y$, 
such that $\sum_y \beta(x,y)=p(x)$ and  $\sum_x \beta(x,y)=q(y)$.
 
Kahn's conjecture \cite[Conjecture~3.6]{FKKK} strengthens this flow formulation as follows.
Let $h:\OO\to\mb R$ be increasing and antipodal, meaning that $h(\ov x)=-h(x)$, and write $h_+=\max(h,0)$.
Consider the Fourier coefficients $\widehat h(S)=\mb E[h(X)\chi_S(X)]$, where $X$ is uniform on $\OO$ and $\chi_S(x)=(-1)^{|S\cap x|}$.
For each non-empty $T \sub [n]$ suppose $\lambda_T(i)\ge0$ for $i\in T$ with $\sum_{i\in T}\lambda_T(i)=\widehat h(T)^2$.
Let $\lambda_i=\sum_{T\ni i}\lambda_T(i)$ and $L_\lambda(x)=\sum_i\lambda_i x_i$.

\begin{theo}[Kahn's flow conjecture]\label{thm:kahn}
For every $h$ and $\lambda_T(i)$ as above, $L_\lambda$ flows upward to $h_+^2$.
\end{theo}

To obtain Kleitman's conjecture, let $h=2\,1_{\mc B}-1$, so $h_+^2=1_{\mc B}$ 
and $\sum_i \lambda_i = \sum_T \widehat h(T)^2 = 1$ by Parseval.
We will motivate our proof of Theorem \ref{thm:kahn} by first giving 
another perspective on the proof of Chv\'atal's conjecture,
also taking the opportunity to sharpen the analytic bounds of \cite{CLL} 
and thus derive some properties of equality cases and stability.
We consider the following equivalent correlation formulation
proposed by Friedgut, Kahn, Kalai and Keller \cite{FKKK}:
if $f,g:\OO\to\{0,1\}$ are increasing and $g$ is antipodal, 
now meaning that $g(x)+g(\ov x)=1$, then $\on{Cov}(f,g)\ge\frac14\min_i\on{Inf}_i[f]$,
where we consider influences $\on{Inf}_i[f]=\mb P(f(X)\ne f(X\triangle\{i\}))$.
To see that this implies Chv\'atal's conjecture, consider any intersecting $\mc I\sub\mc D$, 
and extend $\mc I$ to a maximal intersecting family $\mc B$, which is necessarily increasing and antipodal.
Let $f=1-1_{\mc D}$ and $g=1_{\mc B}$.  Then $\on{Cov}(f,g)=N^{-1}(|\mc D|/2-|\mc D\cap\mc B|)$ 
and  $\on{Inf}_i[f]=2N^{-1}(|\mc D|-2|\mc D(i)|)$, where $N=2^n$. Thus the bound on $\on{Cov}(f,g)$
is exactly the desired bound $|\mc D\cap\mc B|\le\max_i|\mc D(i)|$.

We write  $\cE_T^{(p)}(f)=\mb E|f(X)-f(X\triangle T)|^p$ for $T \sub [n]$ 
and strengthen the required inequality via the estimates
\[ \frac14\min_i\on{Inf}_i[f] = \sum_{T\ne\es} \widehat g(T)^2 \min_i\on{Inf}_i[f]
\le \sum_{T\ne\es} \widehat g(T)^2 \min_{i \in T} \on{Inf}_i[f] \le \mc E_g(f) := \sum_{T\ne\es}\widehat g(T)^2\cE_T^{(2)}(f). \]
Thus it suffices to show $\on{Cov}(f,g)\ge \mc E_g(f)$.
This formulation has not given much away,
as the extremal $g$ has Fourier support on a singleton,
and its advantage is a nice Fourier formula
with a geometric interpretation 
in terms of the two subspace theory of Halmos \cite{Halmos}.
We consider the subspaces $U$ of functions supported on $\mc F=\{f=1\}$ 
and $V$ of functions with Fourier support on $\mc G=\{g=1\}$.
We let $P,Q$ be the orthogonal projections onto $U,V$ and $C=PQP|_U$.
We will show that $\mc E_g(f) = 2N^{-1} \on{Tr}(C-C^2)$.
The arguments in \cite{CLL} can be interpreted as computing 
the dimensions of $0$- and $1$-eigenspaces of $C$,
and Chv\'atal's conjecture comes down to a dimension calculation.
For Kahn's conjecture we employ a more sophisticated version
of this calculation replacing dimension counting 
by singular value estimates
for an appropriate weighted operator.

\section{Geometry of two subspaces}

We start by formulating Halmos' theory, including a proof for the reader's convenience.
Fix subspaces $U,V\sub\mb R^m$. Let $P,Q$ be orthogonal projections onto $U,V$.
For a unit vector $u\in U$, the angle $\theta$ it makes with $V$ satisfies
$\cos^2 \theta = \|Qu\|^2 = \langle u,Qu\rangle = \langle u,Cu\rangle$, where $C:=PQP|_U$. 
Thus the geometry is described by the spectral theory of $C$, as follows.

\begin{prop}\label{prop:halmos}
The $1$- and $0$-eigenspaces of $C$ are $U \cap V$ and $U \cap V^\perp$.
Let \( U_*= U\cap\bigl((U\cap V)\oplus(U\cap V^\perp)\bigr)^\perp\).
Then $U_*$ is $C$-invariant and has an orthonormal eigenbasis
$u_1,\ldots,u_r$ with \(Cu_j=\lambda_j u_j\),
where each \( \lambda_j \in (0,1) \).
Writing $c_j=\sqrt{\lambda_j}$ and $s_j=\sqrt{1-\lambda_j}$, there are
orthonormal vectors $v_1,\ldots,v_r\in U^\perp$ such that the spaces
$L_j=\operatorname{span}\{u_j,v_j\}$ are pairwise orthogonal and are
preserved by both $P$ and $Q$.  In the basis $(u_j,v_j)$,
\begin{equation}\label{eq:halmos-block}
 P=
 \begin{pmatrix}
  1&0\\
  0&0
 \end{pmatrix},
 \qquad
 Q=
 \begin{pmatrix}
  c_j^2&c_js_j\\
  c_js_j&s_j^2
 \end{pmatrix}.
\end{equation}
Thus the corresponding principal angle $\theta_j\in(0,\pi/2)$ is
determined by $c_j=\cos\theta_j$, or equivalently
$\lambda_j=\cos^2\theta_j$.
\end{prop}

\begin{proof}
We note that $C$ is self-adjoint and preserves $U$.
If $u\in U$ is a unit eigenvector with $Cu=\lambda u$ then 
\(  0 \le \lambda
 =\langle u,Cu\rangle
 =\langle u,Qu\rangle
 =\langle Qu,Qu\rangle
 =\|Qu\|^2 \le \|u\|^2 = 1 \).
If $\lambda=0$, then $Qu=0$, and hence $u\in U\cap V^\perp$;
conversely, if $u\in U\cap V^\perp$ then $Cu = PQu=0$.
If $\lambda=1$ then as $\|u\|^2=\|Qu\|^2+\|(I-Q)u\|^2$
and $\|u\|^2=\|Qu\|^2=1$ we have $(I-Q)u=0$, so $Qu=u$ and $u\in U\cap V$. 
Conversely, $C$ is the identity on $U\cap V$.  

Since $C$ is self-adjoint, it preserves $U_*$ and the
spectral theorem  gives an orthonormal eigenbasis $u_1,\ldots,u_r$ 
with eigenvalues  $\lL_1,\ldots,\lL_r$ in $(0,1)$. We define
\[
 v_j=\frac{Qu_j-\lambda_j u_j}{c_js_j}.
\]
Then
\(
 P(Qu_j-\lambda_j u_j)
 =Cu_j-\lambda_j u_j
 =0
\)
and
\(
 \|Qu_j-\lambda_j u_j\|^2
 =
 \|Qu_j\|^2
 -2\lambda_j\langle u_j,Qu_j\rangle
 +\lambda_j^2\|u_j\|^2
 =\lambda_j(1-\lambda_j),
\)
so $v_j$ is a unit vector in $U^\perp$.
Now
\(
 Qu_j=c_j^2u_j+c_js_jv_j
\)
by definition, and applying $Q$ gives
\(
 Qu_j =c_j^2Qu_j+c_js_jQv_j.
\)
Substituting and rearranging gives
\(
 c_js_jQv_j
 =c_j^2s_j^2u_j+c_js_j^3v_j,
\)
so
\(
 Qv_j=c_js_ju_j+s_j^2v_j.
\)
Since also $Pu_j=u_j$ and $Pv_j=0$, both projections preserve
$L_j$, and their matrices are those in \eqref{eq:halmos-block}.

It remains to check for $j \ne k$ that $L_j$ and $L_k$ are orthogonal. 
We already have shown orthogonality between all pairs in $\{u_j,v_j,u_k,v_k\}$
except $(v_j,v_k)$. Note that
\(
 \langle Qu_j,Qu_k\rangle
 =\langle u_j,Qu_k\rangle
 =\lambda_k\langle u_j,u_k\rangle
 =0.
\)
Substituting $Qu_j=\lambda_j u_j+c_js_jv_j$
and $Qu_k=\lambda_k u_k+c_ks_kv_k$
and expanding gives 
\(
 c_js_jc_ks_k\langle v_j,v_k\rangle=0,
\)
so $v_j\perp v_k$.
\end{proof}

\section{Chv\'atal's conjecture}\label{sec:chvatal}

Let $f,g:\OO \to \{0,1\}$ be increasing. Let $\mc F=\{f=1\}$ and $\mc G=\{g=1\}$.
As above, we consider the subspaces $U$ of functions supported on $\mc F=\{f=1\}$ 
and $V$ of functions with Fourier support on $\mc G=\{g=1\}$.
We let $P,Q$ be the orthogonal projections onto $U,V$ and $C=PQP|_U$.

Define the matrix $H$ by $H_{x,S} = N^{-1/2} \chi_S(x)$,
which is the orthogonal change of basis between the standard and Fourier bases.
Write $P_I$ for projection on coordinates $I$.
Then $P=P_{\mc F}$ is diagonal in the standard basis
and  $Q=HP_{\mc G}H$  is diagonal in the Fourier basis.
We have $Q_{y,x}=\frac1N\sum_{S\in\mc G}\chi_S(x)\chi_S(y) =\widehat g(x\triangle y)$.

We connect the energy $\mc E_g(f)$ from the introduction to the Halmos geometry as follows.

\begin{lemma}\label{lem:crossing-mass}
We have $\on{Tr}(C-C^2)=\frac N2\mc E_g(f)$.
\end{lemma}

\begin{proof}
Since $P,Q$ are orthogonal projections,
\[
 \|(I-P)QP\|_{\mathrm{HS}}^2
 =\on{Tr}\bigl(PQ(I-P)QP\bigr)
 =\on{Tr}(PQP)-\on{Tr}(PQPQP)
 =\on{Tr}(C-C^2).
\]
On the other hand,
\[
 \|(I-P)QP\|_{\mathrm{HS}}^2
=  \sum_{x\in\mc F,\,y\notin\mc F}\widehat g(x\triangle y)^2
= \sum_T\widehat g(T)^2|\{x\in\mc F:x\triangle T\notin\mc F\}|.
\]
As $N\cE_T^{(2)}(f)=2|\{x\in\mc F:x\triangle T\notin\mc F\}|$ the lemma follows.
\end{proof}

Next we need the dimensions of $U\cap V^\perp$ and $U\cap V$, the $0$- and $1$-eigenspaces of $C$.
We employ the monomial calculation from \cite[Lemma~3.1]{CLL}.
For $S\sub[n]$, let $M_S(x)=\prod_{i\in S}x_i=1_{\{S\sub x\}}$.  
This is supported on the upset generated by $S$.
Its Fourier support lies in subsets of $S$, as
\begin{equation}\label{eq:monomial-expansion}
 M_S=2^{-|S|}\sum_{T\sub S}(-1)^{|T|}\chi_T.
\end{equation}

\begin{prop}\label{prop:endpoints}
We have $\dim(U\cap V^\perp)=|\mc F\sm\mc G|$ and $\dim(U\cap V)=|\mc F\sm\mc G^\dagger|$,
where $\mc G^\dagger=\OO\sm\{\ov S:S\in\mc G\}$.
\end{prop}

\begin{proof}
As $V=\operatorname{span}\{\chi_T:T\in\mc G\}$
we have $\phi \in V^\perp$ if and only if $\widehat\phi(T)=0$ for every $T\in\mc G$.
Thus $U\cap V^\perp=\ker R_{\mc G}$, where for  $\mc H\sub\OO$ we define
\(R_{\mc H}:U\longrightarrow \mb R^{\mc H}\) 
by \(R_{\mc H}(\phi)=\bigl(\widehat\phi(T)\bigr)_{T\in\mc H}\).

To calculate the dimension, we first note that $\{M_S:S\in\mc F\}$ form a basis of $U$,
as $\mc F$ is increasing and $(M_S(T))$ is triangular. In this basis $R_{\mc H}$
has the matrix with $(T,S)$ entry $\widehat{M_S}(T)$.
If $S\in\mc F\sm\mc G$, then the whole column indexed by $S$ is zero,
as $\mc{G}$ does not contain any subsets of $S$, so
\(
 \operatorname{rank}R_{\mc G}\le |\mc F\cap\mc G|.
\)
For the reverse inequality, consider the submatrix
with rows and columns both indexed by $\mc F\cap\mc G$. 
This is triangular with diagonal $ \widehat{M_S}(S)=2^{-|S|}(-1)^{|S|}\ne0$, so
\(
 \operatorname{rank}R_{\mc G}\ge |\mc F\cap\mc G|.
\)
We deduce 
\(
 \operatorname{rank}R_{\mc G}=|\mc F\cap\mc G|,
\)
so
\(
 \dim(U\cap V^\perp)
 =\dim\ker R_{\mc G}
 =|\mc F\sm\mc G|.
\)

To compute $\dim(U\cap V)$ we consider the map
\(
 J\phi=\chi_{[n]}\phi.
\)
Multiplication by $\chi_{[n]}$ does not change support,
so $J$ is a bijection on $U$.
Also  $\widehat{J\phi}(T)=\widehat\phi(\ov T)$,
so $\phi$ has Fourier support in $\mc G$ if
and only if $J\phi$ has Fourier support in $\OO\sm\mc G^\dagger$,
which is equivalent to $R_{\mc G^\dagger}(J\phi)=0$.
Thus $J$ gives a bijection from $U\cap V$ to $\ker R_{\mc G^\dagger}$, and we conclude that
\(
 \dim(U\cap V)
 =\dim\ker R_{\mc G^\dagger}
 =|\mc F\sm\mc G^\dagger|.
\)
\end{proof}

We also need the following simple  properties of $\cE_T^{(p)}(f)=\mb E|f(X)-f(X\triangle T)|^p$.

\begin{lemma}\label{lem:flip-energy}
Let $f:\OO\to\mb R$ be increasing.  For $p=1,2$ and $\es \ne T\sub[n]$ we have
$\cE_T^{(p)}(f)\ge\max_{i\in T}\cE_i^{(p)}(f)$.  Moreover
$\cE_T^{(2)}(f)=4\sum_{|S\cap T|\text{ odd}}\widehat f(S)^2$ and
$\on{Cov}(f,x_i)=\frac14\cE_i^{(1)}(f)$.  If $f$ is Boolean then
$\cE_i^{(1)}(f)=\cE_i^{(2)}(f)=\on{Inf}_i[f]$.
\end{lemma}

\begin{proof}
Fix $j\notin T$ and let $Y$ be a uniform random subset of $[n]\sm\{j\}$.  Given $Y=y$, write
\[
 a=f(y),\qquad b=f(y\cup\{j\}),\qquad
 c=f(y\triangle T),\qquad d=f((y\triangle T)\cup\{j\}).
\]
Since $f$ is increasing, $a\le b$ and $c\le d$.  Conditioning first on $Y$ gives
\[
 \cE_T^{(p)}(f)
 =\mb E_Y\frac{|a-c|^p+|b-d|^p}{2},\qquad
 \cE_{T\cup\{j\}}^{(p)}(f)
 =\mb E_Y\frac{|a-d|^p+|b-c|^p}{2}.
\]
Thus $\cE_{T\cup\{j\}}^{(2)}(f)-\cE_T^{(2)}(f) = \mb{E}_Y (b-a)(d-c) \ge 0$.
Similarly $\cE_{T\cup\{j\}}^{(1)}(f)-\cE_T^{(1)}(f) \ge0$, as the order-preserving
matching $a \lra c$, $b \lra d$ minimises total $L^1$ distance.
Iterating over the coordinates of $T$ gives $\cE_T^{(p)}(f)\ge\max_{i\in T}\cE_i^{(p)}(f)$.

The Fourier formula follows from
\(
 f(x)-f(x\triangle T)
 =2\sum_{|S\cap T|\text{ odd}}\widehat f(S)\chi_S(x)
\)
and Parseval.
Finally, writing $\mu_j=\mb E[f(X)\mid X_i=j]$ for $j=0,1$
we have  $\on{Cov}(f,x_i)= \mb E[f(X)x_i]-\mb E f\mb E x_i
= \frac{\mu_1}{2} - \frac{\mu_0+\mu_1}{4}
= \frac{\mu_1-\mu_0}{4}$.
Here $\cE_i^{(1)}(f)=\mu_1-\mu_0$ by monotonicity
and in the Boolean case $\cE_i^{(1)}(f)=\cE_i^{(2)}(f)=\on{Inf}_i[f]$.
\end{proof}

The Halmos picture now gives the FKKK correlation reformulation of Chv\'atal.

\begin{cor}\label{cor:antipodal}
If $f,g:\OO\to\{0,1\}$ are increasing and $g$ is antipodal then $\on{Cov}(f,g)\ge\frac14\min_i\on{Inf}_i[f]$.
\end{cor}

\begin{proof}
Let $\eta=\min_i\on{Inf}_i[f]$. Applying Lemma~\ref{lem:crossing-mass}, Lemma~\ref{lem:flip-energy} and Parseval gives
\[
 \on{Tr}(C-C^2)=\frac N2\mc E_g(f)
 =\frac N2\sum_{T\ne\es}\widehat g(T)^2\cE_T^{(2)}(f) 
 \ge \frac N2\sum_{T\ne\es}\widehat g(T)^2 \eta = \eta N/8.
\]
On the other hand, antipodality gives $\mc G^\dagger=\mc G$, so Proposition~\ref{prop:endpoints} 
shows that the $1$- and $0$-eigenspaces of $C$ both have dimension $|\mc F\sm\mc G|$.  
Since $\dim U=|\mc F|$ and $|\mc G|=N/2$, the remaining subspace $U_*$ has dimension
\[
 m=|\mc F|-2|\mc F\sm\mc G|
 =2|\mc F\cap\mc G|-|\mc F|
 =2N\on{Cov}(f,g).
\]
By Proposition~\ref{prop:halmos}, 
on $L_j$ the compression $C$ has eigenvalue $\lambda_j=\cos^2\theta_j$, so 
\[
 \on{Tr}(C-C^2)=\sum_{j=1}^m\cos^2\theta_j\sin^2\theta_j
 \le\frac m4=\frac N2\on{Cov}(f,g).
\]
Comparing the bounds gives the result.
\end{proof}

\section{The sharp correlation inequality}\label{sec:sharp-correlation}

In this section we sharpen the analytic approach and thus derive properties
of the equality cases and stability in Chv\'atal's conjecture.
For any Boolean function $r$ on $\OO$, write $r^\dagger(x)=1-r(\ov x)$.  
If $r=1_{\mc R}$, this is the indicator of $\mc R^\dagger=\{x:\ov x\notin\mc R\}$.  
Throughout this section $f,g:\OO\to\{0,1\}$ are increasing, $\mc F=\{f=1\}$, $\mc G=\{g=1\}$, and 
\[
 a=\on{Cov}(f,g),\qquad b=\on{Cov}(f,g^\dagger),\qquad
 W(f,g)=\sum_{T\ne\es}\widehat g(T)^2\max_{i\in T}\on{Inf}_i[f],\qquad
 \eta=\min_i\on{Inf}_i[f].
\]
Note that $a,b\ge0$ by Harris' inequality.
We will sharpen the inequality $W(f,g)\le \frac{2ab}{a+b}$ for $a+b>0$ from \cite{CLL}.

\begin{prop}\label{prop:angles}
Assume $a+b>0$ and put $\bar\lambda=\tfrac{b}{a+b}$. Then
\begin{equation}\label{eq:CLL-deficit}
 \frac{2ab}{a+b}-W(f,g)
 =\frac2N\sum_{j=1}^m(\lambda_j-\bar\lambda)^2
 +\sum_{T\ne\es}\widehat g(T)^2\bigl(\cE_T^{(2)}(f)-\max_{i\in T}\on{Inf}_i[f]\bigr).
\end{equation}
Thus $W(f,g)\le \frac{2ab}{a+b}$, with equality if and only if $\lambda_j=\bar\lambda$ for every $j$ 
and $\cE_T^{(2)}(f)=\max_{i\in T}\on{Inf}_i[f]$ whenever $\widehat g(T)\ne0$.
\end{prop}

\begin{proof}
Let $d_0$ and $d_1$ be the multiplicities of the eigenvalues $0$ and $1$ of $C$.  Proposition~\ref{prop:endpoints} gives
\[
 d_0=|\mc F|-|\mc F\cap\mc G|,
 \qquad
 d_1=|\mc F|-|\mc F\cap\mc G^\dagger|.
\]
As $Na=|\mc F\cap\mc G|-\frac{|\mc F||\mc G|}{N}$
and $Nb=|\mc F\cap\mc G^\dagger|-\frac{|\mc F||\mc G^\dagger|}{N}$,
we have $N(a+b)=|\mc F\cap\mc G|+|\mc F\cap\mc G^\dagger|-|\mc F| = |\mc F|-d_0-d_1$,
so $\dim U_*=m=N(a+b)$. For each $x \in \OO$ we have
\(
 Q_{x,x}=\frac1N\sum_{S\in\mc G}\chi_S(x)^2=\frac{|\mc G|}{N}.
\)
so $\on{Tr}C=|\mc F||\mc G|/N$. Subtracting $d_1$ from the $1$-eigenspace gives
\(
 \sum_{j=1}^m\lambda_j
 =\frac{|\mc F||\mc G|}{N}-d_1
 =N b.
\)
Thus $\bar\lambda=(\sum_j\lambda_j)/m=\tfrac{b}{a+b}$.

To prove \eqref{eq:CLL-deficit}, we write the left hand side as
\[
 \frac{2ab}{a+b}-W(f,g)
 =\left(\frac{2ab}{a+b}-\mc E_g(f)\right)
  +\left(\mc E_g(f)-W(f,g)\right).
\]
Here the second bracket is the second term on the right hand side by definition,
while the first equals $\frac2N\sum_{j=1}^m(\lambda_j-\bar\lambda)^2$,
via a short calculation using $m=N(a+b)$, $\sum_j\lambda_j=Nb$ and
$ \mc E_g(f)=\frac2N\sum_{j=1}^m\lambda_j(1-\lambda_j)$ by Lemma~\ref{lem:crossing-mass}
and Proposition~\ref{prop:halmos}. Both terms are nonnegative, proving the inequality and the stated equality criterion.
\end{proof}

In the trivial case $a+b=0$, the proof above shows $m=0$ by dimension, so $b=0$ by trace, and so $a=0$.
Thus $C$ only has eigenvalues $0$ or $1$ and $\mc E_g(f)=0$, so  $W(f,g)=0$.

Now suppose that $g$ is antipodal.  Then $g^\dagger=g$, so $a=b$.  If $a>0$ then Proposition~\ref{prop:angles} with $\bar\lambda=1/2$ gives
\[
 a-W(f,g)=\frac2N\sum_{j=1}^m\left(\lambda_j-\frac12\right)^2
 +\sum_{T\ne\es}\widehat g(T)^2\left(\cE_T^{(2)}(f)-\max_{i\in T}\on{Inf}_i[f]\right).
\]
Also, as $g$ is antipodal, $\sum_{T\ne\es}\widehat g(T)^2=1/4$, so
\[
 W(f,g)-\frac{\eta}{4}
 =\sum_{T\ne\es}\widehat g(T)^2\left(\max_{i\in T}\on{Inf}_i[f]-\eta\right).
\]
Adding these identities gives
\begin{equation}\label{eq:antipodal-deficit}
 \eps:=a-\frac{\eta}{4}
 =\frac2N\sum_{j=1}^m\left(\lambda_j-\frac12\right)^2
 +\sum_{T\ne\es}\widehat g(T)^2\bigl(\cE_T^{(2)}(f)-\eta\bigr).
\end{equation}
This also holds in the trivial case $a=0$, where both sides are zero.

\begin{theo}\label{thm:antipodal-stability}
Let $f,g:\OO\to\{0,1\}$ be increasing, with $g$ antipodal, and use the notation above.
\begin{enumerate}
\item Equality in Corollary~\ref{cor:antipodal} holds if and only if every $\lL_j=1/2$ and $\cE_T^{(2)}(f)=\eta$ whenever $T\ne\es$ and $\widehat g(T)\ne0$.
\item For $\delta>0$, let $J_\delta=\{i:\on{Inf}_i[f]<\eta+\delta\}$.  Then
\begin{equation}\label{eq:fourier-stability}
 \left\|g-\mb E[g\mid X_{J_\delta}]\right\|_2^2
 =\sum_{T\not\sub J_\delta}\widehat g(T)^2\le\frac{\eps}{\delta}.
\end{equation}
Consequently there is an increasing Boolean $J_\delta$-junta $g_\delta$ such that $\mb P(g\ne g_\delta)\le2\eps/\delta$.
\item If  $\gamma=\min_{j:j\ne i}(\on{Inf}_j[f]-\eta)>0$ then $\mb P(g(X)\ne X_i)\le2\eps/\gamma$. 
In particular, $\eps=0$ forces $g(x)=x_i$.
\end{enumerate}
\end{theo}

\begin{proof}
Statement 1 is immediate as both terms on the right of \eqref{eq:antipodal-deficit} are nonnegative. 

For 2, note that if $T\not\sub J_\delta$, then some $i\in T$ has $\on{Inf}_i[f]\ge\eta+\delta$, 
so $\cE_T^{(2)}(f)-\eta\ge\delta$ by Lemma~\ref{lem:flip-energy} and
 $\sum_{T\not\sub J_\delta}\widehat g(T)^2\le\eps/\delta$
 from the second term of \eqref{eq:antipodal-deficit}.
 Conditional expectation onto the coordinates $J_\delta$ keeps exactly the Fourier coefficients 
 indexed by subsets of $J_\delta$, proving \eqref{eq:fourier-stability}.
We define the increasing  $J_\delta$-junta $g_\delta$ by $g_\delta(x)=1_{\{\pi(x_{J_\delta})\ge1/2\}}$,
where $ \pi(y)=\mb E[g(X)\mid X_{J_\delta}=y]$, i.e.~we round $\pi$ to $0$ or $1$.
This gives conditional error $\min\{\pi(y),1-\pi(y)\} \le 2\pi(y)(1-\pi(y)) = 2\mb E[(g-\pi(y))^2\mid X_{J_\delta}=y]$,
so $\mb P(g\ne g_\delta) \le 2\left\|g-\mb E[g\mid X_{J_\delta}]\right\|_2^2 \le\frac{2\eps}{\delta}$.

For (3), by (2) we have $J_\gamma=\{i\}$ and $w:=\sum_{T\not\sub\{i\}}\widehat g(T)^2\le\eps/\gamma$.  
Antipodality gives $\mb Eg=1/2$, hence Parseval gives $\widehat g(\{i\})^2=1/4-w$.  Also
\(\widehat g(\{i\})=\frac12(\mb E[g\mid X_i=0]-\mb E[g\mid X_i=1])\le0\)
by monotonicity.  Since $X_i=(1-\chi_{\{i\}})/2$, Parseval gives
\(
 \mb P(g\ne X_i)=\|g-X_i\|_2^2=w+\left(\frac12-\sqrt{\frac14-w}\right)^2.
\)
Here $0\le w\le1/4$ and $0\le\frac12-\sqrt{1/4-w}\le2w$, so the squared term is at most $w$.  Therefore $\mb P(g\ne X_i)\le2w\le2\eps/\gamma$.
\end{proof}

\begin{cor}[Extremizers for Chv\'atal]\label{cor:chvatal-equality}
Let $\mc D\subseteq2^{[n]}$ be a downset. Let $M=\max_i|\mc D(i)|$ and $J=\{i:|\mc D(i)|=M\}$.  
If $\mc I\sub\mc D$ is intersecting and $|\mc I|=M$ then $\mc I=\mc D\cap\mc B$ 
for some maximal intersecting $J$-junta $\mc B\subseteq2^{[n]}$.
In particular, if $|J| \le 2$ then any maximum intersecting subfamily of $\mc D$
is one of the largest stars.
\end{cor}

\begin{proof}
Extend $\mc I$ to a maximal intersecting $\mc B\subseteq2^{[n]}$, and put $f=1-1_{\mc D}$ and $g=1_{\mc B}$.  
Then $\mc I=\mc D\cap\mc B$ and we have equality in Corollary~\ref{cor:antipodal}.
As $\on{Inf}_i[f]=2N^{-1}(|\mc D|-2|\mc D(i)|)$, the minimum-influence coordinates of $f$ are exactly the elements of $J$.
If $T\ne\es$ and $\widehat g(T)\ne0$ then Theorem~\ref{thm:antipodal-stability} and Lemma~\ref{lem:flip-energy} 
gives $\eta=\cE_T^{(2)}(f)\ge\max_{i\in T}\on{Inf}_i[f]$, so $T\sub J$.  Thus $g(x)$ depends only on $x\cap J$,
so $\mc B$ is a $J$-junta. The $|J| \le 2$ conclusion is immediate.
\end{proof}

\section{Kahn's flow conjecture}\label{sec:kahn-reduction}

Let $h:\OO\to\mb R$ be increasing and antipodal. Write $h_+=\max(h,0)$.
For each non-empty $T \sub [n]$ suppose $\lambda_T(i)\ge0$ for $i\in T$ with $\sum_{i\in T}\lambda_T(i)=\widehat h(T)^2$.
Let $\lambda_i=\sum_{T\ni i}\lambda_T(i)$ and $L_\lambda(x)=\sum_i\lambda_i x_i$.
We need to show that  $L_\lambda$ flows upward to $h_+^2$.
We will use the following standard duality,
including the proof for the reader's convenience.

\begin{lemma}\label{lem:flow-cut}
Let $p,q:\OO \to \mb{R}_{\ge 0}$. Then $p$ flows upward to $q$ if and only if 
\(\sum_xp(x)=\sum_xq(x)\) and \(\sum_{x\in\mc D}p(x)\ge\sum_{x\in\mc D}q(x)\) for every downset $\mc D\sub\OO$.
\end{lemma}

\begin{proof}
Necessity is clear. For sufficiency, we consider a transportation network with source $s$, sink $t$, two copies (left and right) of $\OO$, with the following edges and capacities:
connect $s$ to each left $x$ with capacity $p(x)$, connect each right $y$ to $t$ with capacity $q(y)$, connect each left $x$ and right $y$ with $x \sub y$ with infinite capacity.
Then $p$ flows upward to $q$ if and only if there is a flow of value \(\sum_xp(x)=\sum_xq(x)\).
By max-flow/min-cut, this is exactly when \(q(B)\le p(B^\da)\) for every $B\sub\OO$, where $B^\da=\{x:x\sub y\text{ for some }y\in B\}$.
  As $B^\da$ is decreasing and contains $B$ we have \(q(B)\le q(B^\da)\le p(B^\da)\), as required.
\end{proof}

We will now reduce the proof of Theorem~\ref{thm:kahn} to the following inequality,
which is the general form of the FKKK reformulation of Chv\'atal's conjecture.

\begin{theo}\label{thm:weighted-corr-boolean}
Let $f:\OO\to\{0,1\}$ be increasing and let $h:\OO\to\mb R$ be increasing and antipodal.  Then $\on{Cov}(f,h_+^2)\ge\frac14\mc E_h(f)$.
\end{theo}

\begin{proof}[Deduction of Theorem~\ref{thm:kahn} from Theorem~\ref{thm:weighted-corr-boolean}]
Suppose $\lambda_T(i)\ge0$ for $i\in T$ with \(\sum_{i\in T}\lambda_T(i)=\widehat h(T)^2\).
Let $\lambda_i=\sum_{T\ni i}\lambda_T(i)$ and $L_\lambda(x)=\sum_i\lambda_i x_i$.
Let $f:\OO\to\mb R$ be increasing.  By subtracting $\min f$ we may assume $f\ge0$. Write $f(x)=\int_0^\infty f_t(x)\,dt$ with $f_t=1_{\{f>t\}}$.  
Each $f_t$ is increasing, so Theorem~\ref{thm:weighted-corr-boolean} gives $\on{Cov}(f_t,h_+^2)\ge\frac14\mc E_h(f_t)$.  
For each $T \sub [n]$, since $f_t$ is Boolean we have
\[  \int_0^\infty \cE_T^{(2)}(f_t)\,dt = \mb E_X\int_0^\infty
   \bigl|1_{\{f(X)>t\}}-1_{\{f(X\triangle T)>t\}}\bigr|^2\,dt
  =  \mb E_X\,|f(X)-f(X\triangle T)| =\cE_T^{(1)}(f). \]
  Lemma~\ref{lem:flip-energy} gives $\on{Cov}(f,x_i)=\cE_i^{(1)}(f)/4$ and $\cE_T^{(1)}(f)\ge\cE_i^{(1)}(f)$ for $i\in T$, so
\[
 \on{Cov}(f,h_+^2)
 \ge\frac14\sum_T\widehat h(T)^2\cE_T^{(1)}(f)
 \ge\frac14\sum_T\sum_{i\in T}\lambda_T(i)\cE_i^{(1)}(f)
 =\on{Cov}(f,L_\lambda).
\]
As $h$ is antipodal, we have $\sum_i\lambda_i =\sum_{T\ne\es}\widehat h(T)^2 =\mb E h^2$ by Parseval,
so $\mb E L_\lambda=\frac12\mb E h^2 = \mb E h_+^2$. Thus we can rewrite the covariance inequality 
as \(\mb E[f h_+^2]\ge\mb E[fL_\lambda].\)
Applying this to $f=1_{\OO\sm\mc D}$ for any downset $\mc D$
gives $\sum_{x\in\mc D}L_\lambda(x)\ge\sum_{x\in\mc D}h_+(x)^2$,
which suffices by Lemma~\ref{lem:flow-cut}.
\end{proof}

\section{The weighted correlation inequality}\label{sec:weighted-correlation}

It remains to prove Theorem~\ref{thm:weighted-corr-boolean}.  
Let $f:\OO\to\{0,1\}$ be increasing and let $h:\OO\to\mb R$ be increasing and antipodal.  
Write $\mc F=\{f=1\}$ and $\mc P=\{x:h(x)>0\}$. Then $\mc P$ is increasing and intersecting.
Extend $\mc P$ to a maximal intersecting family $\mc A$. Then $\mc A$ is increasing and antipodal.

Put $\mc E=\{S:|S|\text{ even}\}$ and $\mc O=\{S:|S|\text{ odd}\}$.  Define the $\mc E\times\mc O$ matrix $K$ by
\[
 K_{S,T}=2\widehat{h_+}(S\triangle T),\qquad S\in\mc E,\ T\in\mc O.
\]
We will prove
\begin{equation}\label{eq:weighted-deficit}
 N\left(\on{Cov}(f,h_+^2)-\frac14\mc E_h(f)\right)
 =\|K_{\mc E\sm\mc F,\,\mc O\sm\mc F}\|_F^2
 -\sum_{x\in\mc A\sm\mc F}h_+(x)^2
\end{equation}
and
\begin{equation}\label{eq:common-flag-bound}
 \|K_{\mc E\sm\mc F,\,\mc O\sm\mc F}\|_F^2
 \ge \sum_{x\in\mc A\sm\mc F}h_+(x)^2.
\end{equation}
Thus the left side of \eqref{eq:weighted-deficit} is nonnegative, which is exactly Theorem~\ref{thm:weighted-corr-boolean}.  We now prove the two ingredients.

Recall the change of basis matrix $H_{x,S} = N^{-1/2} \chi_S(x)$ from Section~\ref{sec:chvatal}. Let
\[
 V_{\mc E}=\sqrt2\,H_{\mc A,\mc E},\qquad
 V_{\mc O}=\sqrt2\,H_{\mc A,\mc O},\qquad
 D=\operatorname{diag}(h_+(x):x\in\mc A),\qquad
 \rho:=\mb E h^2=\frac2N\sum_{x\in\mc A}h_+(x)^2.
\]

\begin{lemma}\label{lem:weighted-matrix}
The matrices $V_{\mc E}$ and $V_{\mc O}$ are orthogonal, with
\(
 K=V_{\mc E}^{\mathsf T}DV_{\mc O}
\) and
\[
 \|K_{\mc E\sm\mc F,\,\mc O\sm\mc F}\|_F^2
 -\|K_{\mc F\cap\mc E,\,\mc F\cap\mc O}\|_F^2
 =\frac{N-2|\mc F|}{2}\rho.
\]
\end{lemma}

\begin{proof}
As $\mc A$ is antipodal, $|\mc A|=N/2=|\mc E|=|\mc O|$, so $V_{\mc E}$ and $V_{\mc O}$ are square.  When $|S|,|T|$ have the same parity
we have $\sum_{x\in\mc A}\chi_{S\triangle T}(x)
 =\frac12\sum_{x\in\OO}\chi_{S\triangle T}(x)
 =\frac N2\,1_{\{S=T\}}$. Hence $V_{\mc E}$ and $V_{\mc O}$ are orthogonal.  As $h_+$ vanishes outside $\mc A$, for $S\in\mc E$ and $T\in\mc O$,
\[
 (V_{\mc E}^{\mathsf T}DV_{\mc O})_{S,T}
 =\frac2N\sum_{x\in\mc A}h_+(x)\chi_{S\triangle T}(x)
 =2\widehat{h_+}(S\triangle T)=K_{S,T}.
\]
Since $KK^{\mathsf T}=V_{\mc E}^{\mathsf T}D^2V_{\mc E}$, for every $S\in\mc E$,
\[
 (KK^{\mathsf T})_{S,S}
 =\frac2N\sum_{x\in\mc A}h_+(x)^2\chi_S(x)^2
 =\frac2N\sum_{x\in\mc A}h_+(x)^2=\rho.
\]
Thus every row of $K$ has squared norm $\rho$, and similarly for columns.
For the final identity, we write the left hand side as $(A-B)/2$,
where $A$ is the sum of $\|x\|^2$ over rows $x$ in $\mc E\sm\mc F$ and columns $x$ in $\mc O\sm\mc F$
and $B$ is the sum of $\|x\|^2$ over rows $x$ in $\mc F\cap\mc E$ and columns $x$ in $\mc F\cap\mc O$;
this is valid as the off-diagonal blocks cancel. As $A=(N-|\mc F|)\rho$ and $B=|\mc F|\rho$ the identity follows.
\end{proof}

\paragraph{Proof of \eqref{eq:weighted-deficit}.}
By antipodality $h(x)=h_+(x)-h_+(\ov x)$.  Since
$\chi_T(\ov x)=(-1)^{|T|}\chi_T(x)$ we have
 $\widehat h(T)=(1-(-1)^{|T|})\widehat{h_+}(T)$
which is  $2\widehat{h_+}(T)$ if $|T|$ is odd or $0$ otherwise.
By Parseval  $\rho=\sum_T\widehat h(T)^2
 =4\sum_{T\in\mc O}\widehat{h_+}(T)^2$.

Let $\mc{F}=\{f=1\}$ and $A_T=|\{S\in\mc F:S\triangle T\in\mc F\}|$. 
Then $\cE_T^{(2)}(f)=\frac{2}{N}(|\mc F|-A_T)$. Thus
\[
 \frac N4\mc E_h(f)
 =\frac N4\sum_T\widehat h(T)^2\cE_T^{(2)}(f)\\
 =2\sum_{T\in\mc O}\widehat{h_+}(T)^2(|\mc F|-A_T)\\
 =\frac{|\mc F|}{2}\rho
   -2\sum_{T\in\mc O}\widehat{h_+}(T)^2A_T.
\]
There are $A_T/2$ pairs $(R,S)\in(\mc F\cap\mc E)\times(\mc F\cap\mc O)$ with $R\triangle S=T$, so
\[
 2\sum_{T\in\mc O}\widehat{h_+}(T)^2A_T
 =4\sum_{R\in\mc F\cap\mc E}\sum_{S\in\mc F\cap\mc O}
     \widehat{h_+}(R\triangle S)^2\\
 =\sum_{R\in\mc F\cap\mc E}\sum_{S\in\mc F\cap\mc O}K_{R,S}^2
 =\|K_{\mc F\cap\mc E,\,\mc F\cap\mc O}\|_F^2.
\]
Hence
\(
 \frac N4\mc E_h(f)
 =\frac{|\mc F|}{2}\rho
 -\|K_{\mc F\cap\mc E,\,\mc F\cap\mc O}\|_F^2.
\)
Also, as $h_+$ is supported on $\mc A$ and $\mb E h_+^2=\rho/2$
we have $N\on{Cov}(f,h_+^2) = \sum_{x\in\mc A\cap\mc F}h_+(x)^2-\frac{|\mc F|}{2}\rho$,
so
\(
 N\left(\on{Cov}(f,h_+^2)-\frac14\mc E_h(f)\right)
 =\sum_{x\in\mc A\cap\mc F}h_+(x)^2
 -|\mc F|\rho
 +\|K_{\mc F\cap\mc E,\,\mc F\cap\mc O}\|_F^2.
\)
Substituting from Lemma~\ref{lem:weighted-matrix} gives \eqref{eq:weighted-deficit}.

It remains to prove \eqref{eq:common-flag-bound}. Inside $\mb R^{\mc A}$ let
\[
 U=\operatorname{span}\{V_{\mc E}e_S:S\in\mc E\sm\mc F\},\qquad
 V=\operatorname{span}\{V_{\mc O}e_T:T\in\mc O\sm\mc F\}.
\]

\begin{lemma}\label{lem:common-monomial-flag}
For $x\in\mc A\sm\mc F$, let $M_x:\OO\to\{0,1\}$ be the monomial
\(M_x(z)=1_{\{x\sub z\}}\), and let
\(m_x=M_x|_{\mc A}=(M_x(z))_{z\in\mc A}\in\mb R^{\mc A}\).
Then $m_x\in U\cap V$, the vectors $\{m_x:x\in\mc A\sm\mc F\}$ are linearly independent, and
\[
 m_x(z)\ne0\quad\Longrightarrow\quad h_+(z)\ge h_+(x).
\]
Consequently, for every $t\ge0$,
\(
 W_t=\operatorname{span}\{m_x:x\in\mc A\sm\mc F,\ h_+(x)\ge t\}
\)
is a subspace of $U\cap V$ of dimension
\(|\{x\in\mc A\sm\mc F:h_+(x)\ge t\}|\), and every vector in $W_t$ is supported on coordinates $z\in\mc A$ for which $h_+(z)\ge t$.
\end{lemma}

\begin{proof}
By \eqref{eq:monomial-expansion}, each $M_x$ has Fourier support in subsets of $x$,
so is disjoint from $\mc F$ if $x\notin\mc F$, as $\mc F$ is increasing. Let
\(M_x^{\mc E}=\sum_{S\in\mc E}\widehat{M_x}(S)\chi_S \) and
\(M_x^{\mc O}=\sum_{S\in\mc O}\widehat{M_x}(S)\chi_S\).
Then $M_x^{\mc E}|_{\mc A} \in U$ and  $M_x^{\mc O}|_{\mc A}\in V$.
For any $x\in\mc A\sm\mc F$ and $z\in\mc A$, as $\mc{A}$ is intersecting we have $M_x(\ov z)=0$, 
so $M_x^{\mc E}(z)=M_x^{\mc O}(z)=M_x(z)/2$, and so  $m_x\in U\cap V$. As before,
they are linearly independent as their matrix on $\mc A\sm\mc F$ is triangular with diagonal $1$.
If $m_x(z)\ne0$, then $x\sub z$ so monotonicity of $h_+$ gives $h_+(z)\ge h_+(x)$. 
The remaining assertions follow.
\end{proof}

\paragraph{Proof of \eqref{eq:common-flag-bound}.}
Let $P_U$ denote orthogonal projection onto $U$, and set $L=P_UD|_V:V\to U$.  By Lemma~\ref{lem:weighted-matrix}, the vectors $\{V_{\mc E}e_S:S\in\mc E\sm\mc F\}$ and $\{V_{\mc O}e_T:T\in\mc O\sm\mc F\}$ are orthonormal bases of $U$ and $V$, and the matrix of $L$ in these bases is exactly
\(K_{\mc E\sm\mc F,\,\mc O\sm\mc F}\), with $(S,T)$ entry 
\(
 \langle V_{\mc E}e_S,DV_{\mc O}e_T\rangle
 =(V_{\mc E}^{\mathsf T}DV_{\mc O})_{S,T}=K_{S,T}.
\)

Order $\mc A\sm\mc F$ as $x_1,\ldots,x_d$ so that
$h_+(x_1)\ge\cdots\ge h_+(x_d)$.  By Lemma~\ref{lem:common-monomial-flag},
$W_{h_+(x_k)}\sub U\cap V$ has dimension at least $k$ and is supported where $D\ge h_+(x_k)I$.  
Hence for every nonzero $w\in W_{h_+(x_k)}$ we have
\[
 \|Lw\|\,\|w\|\ge\langle Lw,w\rangle
 =\langle P_UDw,w\rangle
 =\langle Dw,w\rangle
 \ge h_+(x_k)\|w\|^2,
\]
The max--min formula for singular values $s_1(L)\ge s_2(L)\ge\cdots$ and Lemma~\ref{lem:common-monomial-flag} give
\[
 s_k(L)=\max_{\dim E=k}\;\min_{0\ne w\in E}\frac{\|Lw\|}{\|w\|} \ge h_+(x_k), 
 \ \text{ so we deduce \eqref{eq:common-flag-bound} by}
\]
\[
 \|K_{\mc E\sm\mc F,\,\mc O\sm\mc F}\|_F^2
 =\|L\|_F^2
 =\sum_k s_k(L)^2
 \ge\sum_{k=1}^d h_+(x_k)^2
 =\sum_{x\in\mc A\sm\mc F}h_+(x)^2.
\]

\subsection*{Statement on AI use}

This proof was found by GPT-6 Astra, following an approach suggested by the author,
for which \cite{CLL} provided the missing piece of the puzzle.
The author simplified and rewrote the proof.

\end{document}